\documentclass[12pt,reqno]{amsart}

\usepackage{amssymb}
\usepackage{amscd}
\usepackage{amsfonts}
\usepackage{setspace}
\usepackage{version}
\usepackage{mathrsfs}

\usepackage{graphicx}

\newtheorem{theorem}{Theorem}[section]
\newtheorem{lemma}[theorem]{Lemma}

\theoremstyle{definition}

\renewcommand{\leq}{\leqslant}
\renewcommand{\geq}{\geqslant}

\newcommand\im{\operatorname{im}}

\def\F{\mathbb{F}}

\def\Z{\mathbf{Z}}
\def\E{\mathbf{E}}
\def\P{\mathbf{P}}

\def\N{\mathbf{N}}

\usepackage{rotating}

\numberwithin{equation}{section}

\begin{document}

\title[S\'ark\"ozy's theorem in function fields]{S\'ark\"ozy's theorem in function fields}


\author{Ben Green}
\address{Mathematical Institute, Radcliffe Observatory Quarter, Woodstock Rd, Oxford OX2 6GG}

\thanks{The author is supported by a Simons Investigator Grant and by ERC Starting Grant 274938, \emph{Approximate Algebraic Structure}. This work was entirely carried out while the author was visiting IMS, Singapore, and he thanks the institute for their hospitality.}

\begin{abstract}
S\'ark\"ozy proved that dense sets of integers contain two elements differing by a $k$th power. The bounds in quantitative versions of this theorem are rather weak compared to what is expected. We prove a version of S\'ark\"ozy's theorem for polynomials over $\F_q$ with polynomial dependencies in the parameters. 

More precisely, let $P_{q,n}$ be the space of polynomials over $\F_q$ of degree $< n$ in an indeterminate $T$.  Let $k \geq 2$ be an integer and let $q$ be a prime power. Set $c(k,q) := (2 k^2  D_q(k)^2\log q)^{-1}$, where $D_q(k)$ is the sum of the digits of $k$ in base $q$. If $A \subset P_{q,n}$ is a set with $|A| > 2q^{(1 - c(k,q))n}$, then $A$ contains distinct polynomials $p(T), p'(T)$ such that $p(T) - p'(T) = b(T)^k$ for some $b \in \F_q[T]$. 
\end{abstract}
\maketitle

\section{Introduction}
Let $q$ be a prime power, and $n$ an integer. Write $P_{q,n}$ for the $n$-dimensional vector space over $\F_q$ consisting of all polynomials $c_0 + c_1 T + \dots + c_{n-1} T^{n-1}$ (where $T$ is an indeterminate) of degree $< n$. 

\begin{theorem}\label{mainthm1}
Let $k \geq 2$ be an integer and let $q$ be a prime power. Set $c(k,q) := (2 k^2 D_q(k)^2\log q)^{-1}$, where $D_q(k)$ is the sum of the digits of $k$ in base $q$. If $A \subset P_{q,n}$ is a set with $|A| > 2q^{(1 - c(k,q))n}$, then $A$ contains distinct polynomials $p(T), p'(T)$ such that $p(T) - p'(T) = b(T)^k$ for some $b \in \F_q[T]$. \end{theorem}

In certain cases, a stronger result than Theorem \ref{mainthm1} is rather trivial. For example, if $k = q^r$ then
\[ (a_0 + a_1 T + \dots + a_m T^m)^{k} = a_0 + a_1 T^k + \dots + a_m T^{mk}.\] Taking $m = \lfloor \frac{n-1}{k} \rfloor$, a bound of the form $|A| \ll q^{(1 - \frac{1}{k}) n}$ follows quickly from the pigeonhole principle. In most cases, however, no trivial argument of this type is available.

Our method also proves the following result, which is qualitatively more general than Theorem \ref{mainthm1} though has a slightly worse exponent.

\begin{theorem}\label{mainthm2}
Let $F \in \F_q[T]$ be a polynomial of degree $k$ with zero constant term, the number of whose roots in $\F_q$ is coprime to $q$. Set $c'(k,q) := (2 k^2 (q - 1)^2 (1 + \log_q k)^2 \log q)^{-1}$. If $A \subset P_{q,n}$ is a set with $|A| > 2 q^{(1 - c'(k,q))n}$, then $A$ contains distinct polynomials $p(T), p'(T)$ such that $p(T) - p'(T) = F(b(T))$ for some $b \in \F_q[T]$. 
\end{theorem}

These results are natural function field analogues of a well-known result of S\'ark\"ozy \cite{sarkozy}, who showed that if $F \in \Z[X]$ is a polynomial with $F(0) = 0$, and if $A \subset \{1,\dots, N\}$ is a set with $|A| \geq N(\log N)^{-c_F}$, then $A$ contains distinct elements $a, a'$ such that $a - a' = F(x)$ for some $x \in \Z$. However, the bounds we obtain in the function field setting are considerably stronger than the natural analogue of S\'ark\"ozy's bound, and correspond to a bound of the form $|A| \geq N^{1 - c_F}$ for the integer version of the problem. No such bound is known, for any polynomial $F$ of degree greater than $1$. 

For the best known bounds on the integer version of the problem, which are a little better than S\'ark\"ozy's bounds, see \cite{ppss}. For previous results on the function field problem see \cite{hoang}. Regarding lower bounds, the only reference we know is an unpublished manuscript by C.~Link \cite{link}, which adapts some constructions of Ruzsa \cite{ruzsa} from the integers. In particular examples are given of sets $A \subset P_{5,2n}$ with $|A| = 5^{3n/2}$ (thus $|A| = N^{3/4}$, where $N = 5^{2n} = |P_{5,2n}|$) such that $A - A$ does not contain any squares other than 0. We thank J.~Wolf for bringing this reference to our attention. 

Theorems \ref{mainthm1} and \ref{mainthm2} are fairly straightforward consequences of the following statement, which is the ``natural'' result output by the methods of this paper.

\begin{theorem}\label{mainthm3}
Let $d \in \N$, and suppose that $0 < m \leq n$. Let $\Phi = (\phi_1,\dots, \phi_n) : \F_q^m \rightarrow \F_q^n$ be a polynomial map, in which each $\phi_i$ is a polynomial over $\F_q$ in $m$ variables with total degree $\leq d$. Suppose that $|\Phi^{-1}(0)|$ is coprime to $q$ \textup{(}and so, in particular, $0 \in \im(\Phi)$\textup{)}. Let $A \subset \F_q^n$ be a set such that $A - A$ intersects $\im (\Phi)$ only at $0$. Then $|A| \leq 2q^{n}e^{-m^2/2nd^2}$.
\end{theorem}

The main tool in proving Theorem \ref{mainthm3} will be a simple -- but previously overlooked -- lemma from the recent breakthrough work of Croot, Lev and Pach \cite{croot-lev-pach}.

\emph{Acknowledgement.} We are grateful to Will Sawin for comments on an earlier version of this note. His comments led to a streamlined version of the argument, giving somewhat better dependencies in the exponents in the main theorems by avoiding the appeal to the Chevalley-Warning theorem. We are also grateful to Lisa Sauermann for a correction to the statement of Theorem \ref{mainthm2}.

\section{Multivariate polynomials}

Every function $f : \F_q^n \rightarrow \F_q$ agrees pointwise with a polynomial
\begin{equation}\label{eq17} P_f(x) =  \sum_{\alpha \in \{0,1,\dots, q-1\}^n} \tilde f(\alpha) x^{\alpha},\end{equation}
where here $x^{\alpha} = x_1^{\alpha_1} \cdots x_n^{\alpha_n}$ (and $x = (x_1,\dots, x_n)$), and $\tilde f(\alpha)$ takes values in $\F_q$. To prove this, note that the space $V$ of functions $f : \F_q^n \rightarrow \F_q$ and the space $W$ of polynomials of the above type are both vector spaces over $\F_q$ of dimension $q^n$. Hence it suffices to show that the natural ``evaluation'' map $V \rightarrow W$ is injective, and this may be done by induction on $n$, using the fact that if a polynomial in one variable of degree $\leq q - 1$ vanishes on $\F_q$ then it is the zero polynomial. 

Define the degree of a polynomial in $W$ to be $\max |\alpha|$, where $|\alpha| = \alpha_1 + \dots + \alpha_n$, and the maximum is taken over those $\alpha$ for which $\tilde f(\alpha) \neq 0$. Note that the degree is always at most $(q - 1)n$.

The notation $\tilde f(\alpha)$ is meant to suggest that we think of this as a kind of ``$\F_q$-valued Fourier transform'' of $f$. We have the following formula for these ``Fourier coefficients'' $\tilde f(\alpha)$ in terms of $f$.

\begin{lemma}\label{fourier-lem}
We have 
\[ \tilde f(\alpha) = \sum_{x \in \F_q^n}  f(x) \prod_{i = 1}^n \sigma_{\alpha_i}(x_i),\] where
\[ \sigma_a(x) = \left\{ \begin{array}{ll} -x^{q - 1 - a} & \mbox{if $a > 0$} \\ 1_{x = 0} & \mbox{if $a = 0$}.  \end{array} \right.\]
\end{lemma}
\begin{proof}
We begin with the right hand side.  Starting with \eqref{eq17}, this may be expanded as 
\[ \sum_{\beta \in \{0,1,\dots, q-1\}^n }\tilde f(\beta) \prod_{i = 1}^n S_{\alpha_i, \beta_i}\]
where
\[ S_{a,b} =  \sum_{x \in \F_q} x^b \sigma_a(x).\] Thus it suffices to show that when $a,b \in \{0,1,\dots, q-1\}$ we have $S_{a,b} = 1_{a = b}$. We have $S_{0,b} = 0^b$ and $S_{q-1, 0} = - \sum_{x \in \F_q} x^0 = 0$, so the claim holds in these cases. Also if $a > 0$ then we have $S_{a,a} = -\sum_{x \in \F_q} x^{q - 1} = -\sum_{x \in \F_q^*} 1 = 1$. In the remaining cases we have $a > 0$, $a \neq b$ and $(a,b) \neq (q-1, 0)$, so $b - a$ is not a multiple of $q - 1$. In all such cases, since $a > 0$ we have
  \[ S_{a,b} = \sum_{x \in \F_q} x^{q - 1 + b - a}.\] If $\lambda \in \F_q^*$ then the multiplication-by-$\lambda$ map is a bijection on $\F_q$, and so $S_{a,b} = \lambda^{q - 1 + b - a}S_{a,b}$. But since $q - 1 + b - a$ is a positive integer and not a multiple of $q - 1$ there is a choice of $\lambda$ for which $\lambda^{q - 1 + b - a} \neq 1$. It follows that $S_{a,b} = 0$. This concludes the proof that $S_{a,b} = 1_{a = b}$ in all cases, and hence the lemma.
\end{proof}

We note that when $q = 2$ this formula takes a particularly simple form, namely
\[ \tilde f(\alpha) = \sum_{x \in \F_2^n}  f(x) \prod_{i : \alpha_i = 0} 1_{x_i = 0}.\] 

\section{$\im(\Phi)$ as the support of a polynomial of low degree}

Using the results of the last section, we establish the following key lemma.

\begin{lemma}\label{lem3-main}
Suppose that $\Phi : \F_q^m \rightarrow \F_q^n$ is a polynomial map of degree $\leq d$, that is to say $\Phi = (\phi_1,\dots, \phi_n)$ with each $\phi_i$ an $m$-variable polynomial over $\F_q$ with total degree $\leq d$. Suppose that $|\Phi^{-1}(0)|$ is coprime to $q$. Then there is a polynomial $P \in \F_q[x_1,\dots, x_n]$, with no variable having degree $> (q - 1)$, such that \begin{enumerate}
\item $\deg P \leq (q - 1)(n - \frac{m}{d})$;
\item $P$ is supported on $\im(\Phi)$ \textup{(}that is, $P(x) = 0$ when $x \notin \im(\Phi)$\textup{)};
\item $P(0) \neq 0$.
\end{enumerate}
\end{lemma}
\begin{proof} Let $f(x) = |\Phi^{-1}(x)|$. Then we will take $P = P_f$ to be the polynomial representation of $f$, as in \eqref{eq17}. Note that $P_f(x) = f(x)$ when evaluated pointwise, and so properties (2) and (3) are immediate. It remains to verify property (1), to which end we must show that $\tilde f(\alpha) = 0$ unless $|\alpha| \leq  (q - 1)(n - \frac{m}{d})$. For this, we use the formula obtained in Lemma \ref{fourier-lem}, which implies that 
\[ \tilde f(\alpha) = \sum_{a \in \F_q^m} \prod_{i = 1}^n \sigma_{\alpha_i}(\phi_i(a)).\] Now observe that $\sigma_{\alpha_i}$ is a polynomial of degree $q - 1- \alpha_i$, this being immediate from the definition when $\alpha_i > 0$, and true also when $\alpha_{i} = 0$ since we have the identity $1_{x = 0} = 1 - x^{q - 1}$. It follows that $\prod_{i = 1}^n \sigma_{\alpha_i}(\phi_i(a))$ is a polynomial in $a = (a_1,\dots, a_m)$ of degree at most $\sum_{i = 1}^n (q - 1 - \alpha_i) d = (n(q - 1) - |\alpha|) d$. If this is $< (q - 1)m$ then, when we sum over $a \in \F_q^m$, we get zero: indeed the lowest-degree monomial with nonzero sum is $\prod_{i = 1}^m a_i^{q - 1}$. Thus indeed $\tilde f(\alpha) = 0$ so long as $(n(q - 1) - |\alpha|) d < m (q - 1)$, which is equivalent to $|\alpha| > (q - 1)(n - \frac{m}{d})$. This concludes the proof.\end{proof}

 \section{Croot--Lev--Pach principle}
 
In this section we complete the proof of Theorem \ref{mainthm3} by applying the idea of Croot, Lev and Pach \cite{croot-lev-pach}.  Suppose that $d, m, \Phi$ are as in the statement of the theorem, and suppose that $A \subset \F_q^n$ is a set for which $A - A$ and $\im(\Phi)$ intersect only at $0$. Let $P$ be the polynomial constructed in Lemma \ref{lem3-main}. Then $P(a - a') = 0$ when $a \neq a'$ (since $P$ is supported on $\im (\Phi)$) but $P(a - a') \neq 0$ when $a = a'$ (since $P(0) \neq 0$). It follows that the rank of the $|A| \times |A|$ matrix $M$ with $M_{a,a'} = P(a - a')$ is precisely $|A|$. On the other hand, it follows from the ideas in \cite{croot-lev-pach} that the rank of $M$ is at most twice the number of indices $\alpha \in \{0,1\dots, q-1\}^n$ with $|\alpha| \leq \frac{1}{2}\deg P$. 

The variant of the proof of this last fact given in \cite{gijswijt} is sufficiently short that we can give it here: we have $M = XYZ$ where
\begin{itemize}
\item $X$ is the $|A| \times q^n$ matrix with $(a,\alpha)$-entry $a^{\alpha}$ (where $\alpha \in\{0,1\dots, q-1\}$);
\item $Y$ is the $q^n \times q^n$ matrix with $(\alpha, \beta)$-entry $c_{\alpha, \beta}$, where $f(x + y) =\sum_{\alpha, \beta} c_{\alpha, \beta} x^{\alpha} y^{\beta}$;
\item $Z$ is the $q^n \times |A|$ matrix with $(\beta, b)$-entry $b^{\beta}$.
\end{itemize}
Thus $\mbox{rk}(M) \leq \mbox{rk}(Y)$. However, we may write
\[ f(x+ y) = \sum_m m(x) a_m(y) + m(y) b_m(x)\] where $m$ ranges over the set  of monomials $m(t) = t^{\gamma}$ with $\gamma \in S$, the set of all indices $\gamma \in \{0,1,\dots, q-1\}$ with $|\gamma| \leq \frac{1}{2}\deg f$, and so $Y$ is the sum of two matrices, one with row rank at most $|S|$ and the other with column rank at most $|S|$.

Returning to our main argument, putting these observations together shows that $|A|$ is at most twice the number of indices $\alpha \in \{0,1,\dots ,q-1\}^n$ with $|\alpha| \leq \frac{1}{2}(q - 1)(n - \frac{m}{d})$, or in other words
\[ |A| \leq 2q^n \P\big(X_1 + \dots + X_n \leq \frac{1}{2}(q - 1)(n - \frac{m}{d})\big),\]
where the $X_i$ are i.i.d. random variables, each having the uniform distribution on $\{0,1,\dots, q-1\}$. Equivalently, 
\begin{equation}\label{a-bound} |A| \leq 2q^n \P \big(Y_1 + \dots + Y_n \leq - \frac{m}{d}\big),\end{equation} where the $Y_i$ are i.i.d. random variables, each uniformly distributed on the points $-1 + j(\frac{2}{q-1})$, $j = 0,1,\dots, q-1$, which all lie in $[-1,1]$.

With some calculation, one could presumably obtain a fairly satisfactory asymptotic for the right-hand side of \eqref{a-bound}. However, this would not be particularly straightforward and in any case, in our opinion, the general behaviour of the exponents in Theorem \ref{mainthm3} is of more interest than any numerical constants. To obtain this we can use the quite general and rather tidy large deviation inequality of Bernstein type due to Hoeffding, namely $\P \big( Y_1 + \dots + Y_n \leq - t) \leq e^{-t^2/2n}$. This applies to any independent random variables $Y_i$ taking values in $[-1,1]$ and with $\E Y_i = 0$. By \eqref{a-bound}, this gives the bound $|A| \leq 2 q^n e^{-m^2/2nd^2}$, which is the same as the one claimed in Theorem \ref{mainthm3}.

\section{S\'arkozy's theorem in function fields}

The remaining task is to deduce Theorems \ref{mainthm1} and \ref{mainthm2} from Theorem \ref{mainthm3}. 

\emph{Proof of Theorem \ref{mainthm1}.} Identify $P_{q,n}$ with $\F_q^n$ via the map $c_0 + c_1T + \dots + c_{n-1} T^{n-1} \mapsto (c_0, \dots, c_{n-1})$. Set $m := \lfloor \frac{n-1}{k}\rfloor + 1$. We have 
\begin{align*} (a_0 & + a_1 T + \dots + a_{m-1} T^{m-1})^k \\ & = g_0(a_0,\dots, a_m) + g_1(a_0,\dots, a_m) T + \dots + g_{n-1}(a_0,\dots, a_m) T^{n-1},\end{align*}
where the $g_i$ are polynomials with $g_i(0) = 0$ and $\deg g_i \leq D_q(k)$, the sum of the digits of $k$ in base $q$. To see this last point, note that if 
 $k = b_0 + b_1 q + \dots + b_r q^r$ in base $q$ then
 \begin{align*} (& a_0 + a_1 T + \dots + a_{m-1} T^{m-1})^k \\ &  = (a_0 + a_1 T + \dots + a_{m-1} T^{m-1})^{b_0} (a_0 + a_1 T^q + \dots + a_{m-1} T^{qm})^{b_1}  \cdots\end{align*}
Define $\Phi : \F_q^m \rightarrow \F_q^n$ by taking $\Phi = (\phi_1,\dots, \phi_n)$, where \[ \phi_i(a_0,\dots, a_{m-1}) = g_{i-1}(a_0,\dots, a_{m-1}).\] Then, under the identification of $P_{q,n}$ with $\F_q^n$, $\im (\Phi)$ is precisely the set of $k$th powers of polynomials. Furthermore $\deg \Phi \leq D_q(k)$, and $\Phi^{-1}(0) = \{0\}$. Thus we may apply Theorem \ref{mainthm3} with $m = \lfloor \frac{n-1}{k}\rfloor + 1 \geq \frac{n}{k}$ and $d = D_q(k)$ and thereby obtain Theorem \ref{mainthm1}.

We remark that Will Sawin has pointed out to us that if $q = p^r$ then one could, if desired, replace $D_q(k)$ by $D_p(k)$. This may be done by considering the vector space of degree $< n$ polynomials over $\F_q$ as a vector space of dimension $nr$ over $\F_p$, noting that the $k$th power map is a polynomial map of degree $D_p(k)$.

\emph{Proof of Theorem \ref{mainthm2}.} This is almost the same as the proof of Theorem \ref{mainthm1}, except that now we can only bound $\deg \Phi$ by $\sup_{k' \leq k} D_q(k') \leq (q - 1)(1 + \log_q k)$. We must also be a little more careful in checking that $|\Phi^{-1}(0)|$ is coprime to $q$. To see this, note that $|\Phi^{-1}(0)|$ is the number of polynomials $p \in \F_q[T]$ with $F(p(T)) \equiv 0$. Since $\deg F(p(T)) = \deg F \deg p$, we must have $\deg p = 0$, that is to say $p(T) = c$ is a constant polynomial. For such a polynomial, $F(p(T))$ is the zero polynomial if and only if $c$ is a root of $F$. 

 Thus $|\Phi^{-1}(0)|$ is the number of roots of $F$ in $\F_q$ which, by assumption, is coprime to $q$.

\end{document}